\documentclass[12pt]{amsart}

\usepackage[all]{xy}
\usepackage{graphicx}
\usepackage{amsmath}
\usepackage{amsfonts}
\usepackage{amssymb}
\usepackage{amsthm}
\usepackage{amscd}
\usepackage{textcomp}
\usepackage{hyperref}
\usepackage{extarrows}

\DeclareFontEncoding{OT2}{}{} 
  \newcommand{\textcyr}[1]{%
    {\fontencoding{OT2}\fontfamily{wncyr}\fontseries{m}\fontshape{n}%
     \selectfont #1}}
\newcommand{\Sha}{{\mbox{\textcyr{Sh}}}}

\newcommand{\cA}{\mathcal{A}}

\newcommand{\ra}{\rightarrow}
\newcommand{\lra}{\longrightarrow}

\newcommand{\del}{\delta}
\newcommand{\tensor}{\otimes}

\newcommand{\ssstyle}{\scriptscriptstyle}

\newcommand{\Q}{\mathbf{Q}}

\newcommand{\G}{{\mathbb G}}

\newcommand{\cX}{\mathcal{X}}
\newcommand{\Z}{\mathbf{Z}}

\newcommand{\kbar}{\overline{k}}

\newcommand{\isom}{\cong}

\renewcommand{\O}{\mathcal{O}}

\DeclareMathOperator{\Br}{Br}

\DeclareMathOperator{\Gal}{Gal}

\DeclareMathOperator{\Ima}{Im}

\DeclareMathOperator{\Pic}{Pic}

\DeclareMathOperator{\Spec}{Spec}

\DeclareMathOperator{\Div}{Div}

\DeclareMathOperator{\Ker}{Ker}

\DeclareMathOperator{\ur}{nr}

\DeclareMathOperator{\Hom}{Hom}
\DeclareMathOperator{\TT}{TT}

\theoremstyle{plain}
\newtheorem{theorem}{Theorem}[section]

\newtheorem{corollary}[theorem]{Corollary}

\theoremstyle{definition}

\theoremstyle{remark}

\newtheorem{remark}[theorem]{Remark}

\numberwithin{equation}{section}

 1

\newcommand{\thetitle}
{Groups of components of N{\'e}ron models of Jacobians and Brauer groups}

\begin{document}

\title{\thetitle}
\author{Saikat Biswas}
\address{Division of Mathematics, University of Minnesota-Morris,
Morris, MN 56267}
\email{biswass@morris.umn.edu}

\begin{abstract}
Let $X$ be a proper, smooth, and geometrically connected curve over a non-archimedean local field $K$. In this paper, we relate the component group of the N{\'e}ron model of the Jacobian of $X$ to the Brauer group of $X$.
\end{abstract}

\maketitle


\section{Introduction}
Let $K$ be a non-archimedean local field. Thus $K$ is a complete discrete valuation field with finite residue field $k$. Let $K^{\ur}$ be the maximal unramified extension of $K$. Let $X$ be a proper, smooth, and geometrically connected curve over $K$, and $X^{\ur}=X\otimes_{K}{K^{\ur}}$ be the corresponding curve over $K^{\ur}$. Let $\del$ and $\del'$ denote, respectively, the \emph{index} and the \emph{period} of $X$, and let $\del^{\ur}$ and ${\del^{\ur}}'$ denote the corresponding quantities associated to $X^{\ur}$.
Let $A$ be the \emph{Jacobian variety} of $X$ over $K$, $\Phi_{\ssstyle{A}}$ the $k$-group scheme of connected components of the N{\'e}ron model of $A$, and $c_{\ssstyle{A}}=\#\Phi_{\ssstyle{A}}(k)$ the corresponding \emph{Tamagawa number} of $A$ at $K$. Consider the Brauer-Grothendieck group $\Br(X)=H^2(X,\G_{m})$. Let $\Br_{0}(X)$ denote the image of $\Br(K)\to\Br(X)$, and $\Br_{\ur}(X)$ denote the kernel of $\Br(X)\to\Br(X^{\ur})$. In this paper, we prove

\begin{theorem}[Main Theorem]\label{main}
There exists an exact sequence
$$0 \to \Hom\big(\Br_{\ur}(X)/\Br_{0}(X),\Q/\Z\big) \to \Phi_{\ssstyle{A}}(k) \to \Z/d\Z \to 0$$
where $d={\del}'/{\del^{\ur}}'$.
\end{theorem}

It follows that

\begin{corollary}\label{mc}
$\Br_{\ur}(X)/\Br_{0}(X)$ is a finite group of order $c_{\ssstyle{A}}/d$.
\end{corollary}

Corollary \ref{mc} has an interesting application. To explain this, consider a \emph{global} field $K$. Thus $K$ is either a finite extension of $\Q$ i.e. a number field, or is finitely generated and of transcendence degree $1$ over a finite field $k$ i.e. a function field. In the number field case, let $U$ denote a nonempty open subscheme of $\Spec{\O_K}$, and in the function field case, let $U$ denote a nonempty open subscheme of the unique, smooth, complete, and irreducible curve $V$ over $k$ whose function field is $K$. Consider a regular, connected scheme $\cX$ of dimension $2$ with a proper morphism $\pi:\cX\to U$ such that its generic fiber $X=\cX\otimes_{U}K$ is a smooth and geometrically connected curve over $K$. Let $S$ be the set of primes of $K$ not corresponding to a point of $U$, and $\overline{K}$ be the separable algebraic closure of $K$. Note that $S$ contains all archimedean primes of $K$ in the number field case, and that it may be empty in the function field case. For each prime $v\not\in{S}$, let $K_v$ denote the completion of $K$ at $v$, and let $X_v=X\otimes_{K}{K_v}$. Let $\del$ and ${\del}'$ be, respectively the index and period of $X$ while $\del_v$ and ${\del_v}'$ be the corresponding quantities associated to $X_v$. It is known that $\del_{v}\neq{1}$ for only finitely many primes $v$, and that either $\del_{v}={\del_{v}}'$ or $\del_{v}=2{\del_{v}}'$ \cite[Theorem 8]{licht:duality}. Let $\Br(\cX)$ denote the Brauer group of $\cX$ and define $\Br(\cX)'$ by the exactness of the sequence
$$0 \to \Br(\cX)' \to \Br(\cX) \to \bigoplus_{v\in{S}}\Br(X_v)$$
Now let $A/K$ be the Jacobian variety of $X$ over $K$, and denote by $\Sha(A/K)$ the \emph{Shafarevich-Tate group} of $A/K$. Generalizing the work of Artin \cite{tate:BSD} and Milne \cite{milne:brauer}, Gonzalez-Aviles has shown \cite{gonzalez:brauer} that

\begin{theorem}[Gonzalez-Aviles]\label{gon}
Suppose that the integers $\del_{v}'$ are pairwise co-prime and that $\Sha(A/K)$ contains no nonzero infinitely divisible elements. Then there is an exact sequence 
$$0 \to T_0 \to T_1 \to \Br(\cX)' \to \Sha(A/K)/T_2 \to T_3 \to 0$$ 
in which $T_0$, $T_1$, $T_2$ and $T_3$ are finite groups of orders
\begin{align*}
\#{T_0} &={\del}/{\del}'\\
\#{T_1} &={2}^{e}\\
\#{T_2} &={\del}'/{\prod{\del_{v}'}}\\
\#{T_3} &=\frac{{\del}'/{\prod{\del_{v}'}}}{2^f}
\end{align*}
where 
$$e=\max (0,d'-1)$$
and
\[
f=
\begin{cases}
1 &\text{if } {\del}'/{\prod{{\del_{v}}'}} \text{is even and } d'\geq 1\\
0 &\text{otherwise}
\end{cases}
\]
Here $d'$ is the number of primes $v$ for which $\del_{v}=2{\del_{v}}'$. In particular, if one of $\Sha(A/K)$ or $\Br(\cX)'$ is finite, then so is the other, and their orders are related by 
$${\del}{\del}'\;\#\Br(\cX)'\;=\;2^{e+f}\;\prod_{v}{({\del_v}')^2}\;\#{\Sha(A/K)}$$
\end{theorem}

Now let $K_v^{\ur}$ be the maximal unramified extension of $K_v$, and $X_{v}^{\ur}=X_{v}\otimes_{K_v}K_v^{\ur}$ be the fiber over $K_v^{\ur}$, with index $\del_{v}^{\ur}$ and period ${\del_{v}^{\ur}}'$. Let $\Br_{0}(X_v)$ be the image of the map $\Br(K_v)\to\Br(X_v)$, and $\Br_{\ur}(X_v)$ be the kernel of the map $\Br(X_v)\to\Br(X_v^{\ur})$. Let $c_{\ssstyle{A,v}}$ the Tamagawa number of $A$ at $v$. Note that Theorem \ref{main} applies to $X_{v}$. Combining Corollary \ref{mc} with Theorem \ref{gon}, we obtain

\begin{corollary}
Suppose that the integers $\del_{v}'$ are pairwise co-prime, and $\Sha(A/K)$ is finite. Then
$$\#{\Sha(A/K)}\;\prod_{v}{c_{\ssstyle{A,v}}}\;=\;M\;\#\Br(\cX)'\;\prod_{v}{\#\big(\Br_{\ur}(X_v)/\Br_{0}(X_{v})\big)}$$
where $M$ is a rational number given by
$$M=\frac{{\del}{\del'}}{2^{e+f}\prod_{v}{{\del_v'}{\del_{v}^{\ur}}'}}$$
\end{corollary}

Of course, the left-hand term in the above formula appears in the statement of the well-known Birch and Swinnerton-Dyer Conjecture.

\begin{remark}
The hypothesis that the integers $\del_{v}'$ are pairwise coprime in Theorem \ref{gon} can be dropped when $K$ is a function field, and $S=\emptyset$. More precisely, assume that the curve $V/k$ introduced above is also geometrically connected, and consider $\cX$ to be a smooth, proper, and geometrically connected surface endowed with a proper and flat morphism $f:\cX\to V$ whose generic fiber is $X\to\Spec{K}$. In this case, it is shown in \cite[Cor. 3]{llr2} that if, for some prime $l$, the $l$-part of the group $\Br(\cX)$ or of the group $\Sha(A/K)$ is finite, then
$${\del}^{2}\;\#\Br(\cX)\;=\;\prod_{v}{\del_{v}}{\del_{v}}'\;\#\Sha(A/K)$$ and $\#\Br(\cX)$ is a square. It follows that, in this case, we get
$$\#{\Sha(A/K)}\;\prod_{v}{c_{\ssstyle{A,v}}}\;=\;N\;\#\Br(\cX)\;\prod_{v}{\#\big(\Br_{\ur}(X_v)/\Br_{0}(X_{v})\big)}$$
where the rational number $N$ is given by
$$N=\frac{{\del}^2}{\prod_{v}{{\del_v}{\del_{v}^{\ur}}'}}$$
\end{remark}

\section*{Acknowledgements}
I thank the referee whose careful comments and detailed suggestions have greatly enhanced the content as well as the presentation of this paper. I also thank Dino Lorenzini for many helpful conversations.

\section{Preliminaries}
\subsection{Component Groups, Tamagawa numbers}
Let $K$ be a complete, discretely valued field with finite residue field $k$. Let $A/K$ be an abelian variety over $K$, and let $\cA$ be the \emph{N{\'e}ron model} \cite{blr} of $A/K$ over $\Spec{\O_K}$. The closed fiber $\cA_{k}$ of $\cA$ is a $k$-group scheme, not necessarily connected. Let $\cA_{k}^{0}$ be the connected component of $\cA_{k}$ containing the identity. Over $\Spec{k}$, there is an exact sequence of group schemes 
\begin{equation}\nonumber
0 \to \cA_{k}^0 \to \cA_{k} \to \Phi_{\ssstyle{A}} \to 0
\end{equation}
where the quotient $\Phi_{\ssstyle{A}}$ is a finite, {\'e}tale group scheme over $k$. Equivalently, $\Phi_{\ssstyle{A}}$ is a finite abelian group with a continuous action of $\Gal({\kbar}/{k})$ on it. The group scheme $\Phi_{\ssstyle{A}}={\cA_{k}}/{\cA_{k}^{0}}$
is called the \emph{component group} of $A$. The group of rational points $\Phi_{\ssstyle{A}}(k)$, called the \emph{arithmetic component group} of $A$, counts the number of connected components of $\cA_{k}$ which are geometrically connected and $c_{\ssstyle{A}}=\#{\Phi_{\ssstyle{A}}(k)}$ is called the \emph{Tamagawa number} of $A/K$. Now let $K^{\ur}$ be the maximal unramified extension of $K$. The inclusion $\Gal({\overline{K}}/K^{\ur})\subset\Gal({\overline{K}}/K)$
induces a map $H^1(K,A) \to H^1(K^{\ur},A)$, whose kernel corresponds to the \emph{unramified} subgroup of $H^1(K,A)$.
The map may also be given as $WC(A/K) \to WC(A/K^{\ur})$ where, $WC(A/K)\isom H^1(K,A)$ denotes the Weil-Ch{\^a}telet group of $A$ over $K$. We denote this kernel by $\TT(A/K)$ and call it the group of \emph{Tamagawa torsors} of $A$ over $K$ \cite{biswas}.

\begin{theorem}\label{tt}
There exists a canonical isomorphism of finite abelian groups
$$\TT(A/K)\isom H^1(k,\Phi_{\ssstyle{A}})$$
\end{theorem}

\begin{proof}
The inflation-restriction sequence
$$0 \lra H^1(K^{\ur}/K,A(K^{\ur})) \lra H^1(K,A) \lra H^1(K^{\ur},A)$$
identifies the set of Tamagawa torsors with the injective image of the group $H^1(K^{\ur}/K,A(K^{\ur}))$ in $H^1(K,A)$. 
There is an isomorphism \cite[\S Prop I.3.8]{milne:duality}
$H^1(K^{\ur}/K,A(K^{\ur}))\isom H^1(k,\Phi_{\ssstyle{A}})$
and the latter group is finite, since $\Phi_{\ssstyle{A}}$ is finite.
\end{proof}

\begin{corollary}\label{ttorder}
Suppose that $A/K$ is a Jacobian variety. Then there exists a canonical perfect pairing of finite abelian groups
$$\TT(A/K) \times \Phi_{A}(k) \lra \Q/\Z$$
In particular, $\TT(A/K)$ has order $c_{\ssstyle{A}}$.
\end{corollary}

\begin{proof}
This follows from Theorem \ref{tt}, the fact that $A$ is a self-dual abelian variety, and the perfectness of the pairing \cite[(4.5)]{mcc} induced by Grothendieck's pairing \cite[(2.1)]{mcc}.
\end{proof}

Thus, for any Jacobian variety $A/K$, there is an isomorphism 
\begin{equation}\label{ttdual}
\Phi_{\ssstyle{A}}(k)\isom \Hom(\TT(A/K),\Q/\Z)
\end{equation}
of finite, abelian groups.

\begin{remark}
Since $\Phi_{\ssstyle{A}}$ is finite, its Herbrand quotient is $1$ which implies that 
$\#H^1(k,\Phi_{\ssstyle{A}})=\#H^0(k,\Phi_{\ssstyle{A}})=c_{\ssstyle{A}}$. Thus, for any abelian variety $A/K$ (and not just Jacobians), it follows from Theorem \ref{tt} that $\#\TT(A/K)=c_{\ssstyle{A}}$.
\end{remark}

\subsection{Picard Groups, Jacobian Varieties, and Brauer Groups}
Let $X$ be a smooth, projective, geometrically connected curve defined over any field $K$. Let $\overline{K}$ be a separable closure of $K$, and let $\overline{X}=X\tensor_{K}\overline{K}$. Let $\Div(\overline{X})$ be the group of divisors of $\overline{X}$ i.e. the free abelian group generated by the points of $X(\overline{K})$. Note that ${\Div(\overline{X})}^{G_K}=\Div(X)$, where $G_K=\Gal(\overline{K}/K)$. There is a natural summation map $\Div(X_K)\to\Z$ whose image is $\del\Z$, where $\del$ is the \emph{index} of $X$. Equivalently, $\del$ is the least positive degree of a divisor in $\Div(X_K)$.
 Let $P={\Pic}_{X}$ be the \emph{Picard scheme} of $X$ so that $P(\overline{K})=\Pic(\overline{X})$. It follows that $P(K)={P(\overline{K})}^{G_{K}}={\Pic(\overline{X})}^{G_K}$. The Picard scheme is a smooth group scheme over $K$ whose identity component $A=\Pic^{0}_{X}$ is called the \emph{Jacobian variety} of $X$. There is an exact sequence of $G_K$-modules
\begin{equation}\label{e0}
0 \to A(\overline{K}) \to P(\overline{K}) \xrightarrow{\textrm{deg}} \Z \to 0
\end{equation}
where deg is the degree map on $\Pic(\overline{X})$. Taking $G_K$-invariants of \ref{e0}, we obtain the exact sequence
\begin{equation}
0 \to A(K) \to P(K) \xrightarrow{\textrm{deg}} {\del}'\Z \to 0
\end{equation}
where ${\del}'$ is the \emph{period} of $X$. Equivalently, $\del'$ is the least positive degree of a divisor class in $P(K)={\Pic(\overline{X})}^{G_K}$. The image of the map $\Div(X)\to P(K)$ is denoted by $\Pic(X)$. Furthermore, it is known that $\Pic(X)=H^1(X,\G_m)$. Let $\Br(X)=H^2(X,\G_m)$ be the Brauer-Grothendieck group of $X$. By Lemma 2.2 in \cite{milne:brauer}, there is an exact sequence
\begin{equation}\label{pb}
0 \to \Pic(X) \to P(K) \to \Br(K) \to \Br(X) \to H^1(K,P) \to 0
\end{equation}
The zero on the right-hand end follows from \cite[Cor.I.4.21]{milne:duality}.

\section{The Main Theorem}
In this section, we prove

\begin{theorem}\label{mt}
Let $X$ be a proper, smooth, geometrically connected curve over a non-archimedean local field $K$ having finite residue field $k$, with index $\del$ and period $\del'$. Let $X^{\ur}=X\otimes_{K}{K^{\ur}}$ be the corresponding curve over $K^{\ur}$, with index $\del^{\ur}$ and period ${\del^{\ur}}'$. Let $\Br_{0}(X)$ denote the image of $\Br(K)\to\Br(X)$, and $\Br_{\ur}(X)$ denote the kernel of the map $\Br(X)\to\Br(X^{\ur})$. Let $A$ be the Jacobian variety of $X$ over $K$. Then there exists an exact sequence 
$$0 \to \Hom\big(\Br_{\ur}(X)/\Br_{0}(X),\Q/\Z\big) \to \Phi_{A}(k) \to \Z/d\Z \to 0$$
where $\Phi_{\ssstyle{A}}$ is the component group of $A$, and $d={\del}'/{\del^{\ur}}'$.
\end{theorem}

\begin{proof}
The short exact sequence
$$0 \to A \to P \to \Z \to 0$$
over $K$ and $K^{\ur}$ gives rise, respectively, to the exact rows of the commutative diagram
\[
\xymatrix{
P(K)\ar[r]\ar[d] &\Z\ar[r]\ar[d] &H^1(K,A)\ar[r]\ar[d] &H^1(K,P)\ar[r]\ar[d] &0\\
P(K^{\ur})\ar[r] &\Z\ar[r] &H^1(K^{\ur},A)\ar[r] &H^1(K^{\ur},P)\ar[r] &0}
\]
The colums in the diagram are induced by the inclusion $K\subset{K^{\ur}}$. The image of the map $P(K)\to\Z$ is, by definition, ${\del}'\Z$ while that of $P(K^{\ur})\to\Z$ is ${\del^{\ur}}'\Z$. Thus we have the diagram
\[
\xymatrix{
0\ar[r] &\Z/{\del}'\Z\ar[r]\ar[d] &H^1(K,A)\ar[r]\ar[d] &H^1(K,P)\ar[r]\ar[d] &0\\
0\ar[r] &\Z/{\del^{\ur}}'\Z\ar[r] &H^1(K^{\ur},A)\ar[r] &H^1(K^{\ur},P)\ar[r] &0}
\]
Since ${\del^{\ur}}'$ divides ${\del}'$, the leftmost vertical map is surjective. The kernel of this map is $\Z/d\Z$ where $d={\del}'/{\del^{\ur}}'$ . The middle vertical map has kernel $H^1(K^{\ur}/K,A(K^{\ur}))\isom\TT(A/K)$ by Theorem \ref{tt}. The kernel of the rightmost vertical map is $H^1(K^{\ur}/K,P(K^{\ur}))$. Snake Lemma then gives an exact sequence
\begin{equation}\label{e1}
0 \to \Z/d\Z \to \TT(A/K) \to H^1(K^{\ur}/K,P(K^{\ur})) \to 0
\end{equation}
We now describe the right-most term in the exact sequence \ref{e1}.
Consider the commutative diagram
\[
\xymatrix{
P(K)\ar[r]\ar[d] &\Br(K)\ar[r]\ar[d] &\Br(X)\ar[r]\ar[d] &H^1(K,P)\ar[r]\ar[d] &0\\
P(K^{\ur})\ar[r] &\Br(K^{\ur})\ar[r] &\Br(X^{\ur})\ar[r] &H^1(K^{\ur},P)\ar[r] &0}
\]
where the top and bottom row are obtained by applying \ref{pb} to $X$ and $X^{\ur}$ respectively. Since $\Br(K^{\ur})=0$ \cite{serre:localfields}, the above diagram reduces to
\[
\xymatrix{
0\ar[r] &\Br_{0}(X)\ar[r]\ar[d] &\Br(X)\ar[r]\ar[d] &H^1(K,P)\ar[r]\ar[d] &0\\
&0\ar[r] &\Br(X^{\ur})\ar[r] &H^1(K^{\ur},P)\ar[r] &0}
\]
Snake Lemma then yields the exact sequence
$$0 \ra \Br_{0}(X) \ra \Br_{\ur}(X) \ra H^1(K^{\ur}/K,P(K^{\ur})) \ra 0$$
which yields the isomorphism 
$$\Br_{\ur}(X)/\Br_{0}(X)\isom H^1(K^{\ur}/K,P(K^{\ur}))$$
The exact sequence \ref{e1} can now be given as
\begin{equation}\label{e2}
0 \to \Z/d\Z \to \TT(A/K) \to \Br_{\ur}(X)/\Br_{0}(X) \to 0
\end{equation}
Dualizing this sequence, we obtain
\begin{equation}\label{e3}
0 \to {\big(\Br_{\ur}(X)/\Br_{0}(X)\big)}^{\vee} \to {\TT(A/K)}^{\vee} \to {\Z/d\Z}^{\vee} \to 0
\end{equation}
Here $M^{\vee}=\Hom(M,\Q/\Z)$ denotes the Pontryagin dual of $M$. By \ref{ttdual}, ${\TT(A/K)}^{\vee}\isom\Phi_{\ssstyle{A}}(k)$. On the other hand, the homomorphism $\alpha:1\mapsto\big(\frac{1}{d}+\Z\big)$ has order $d$ and generates $\Hom(\Z/d\Z,\Q/\Z)$, so that $\Hom(\Z/d\Z,\Q/\Z)\isom\Z/d\Z$.
Thus \ref{e3} can be given as
\begin{equation}\nonumber
0 \to \Hom\big(\Br_{\ur}(X)/\Br_{0}(X),\Q/\Z\big) \to \Phi_{\ssstyle{A}}(k) \to \Z/d\Z \to 0
\end{equation}
\end{proof}

The following corollary is immediate

\begin{corollary}\label{cor1}
The quotient group $\Br_{\ur}(X)/\Br_{0}(X)$ is finite of order $c_{\ssstyle{A}}/d$.
\end{corollary}

\begin{remark}
By \cite[Prop 2.1]{thelene:skoro}(which applies to any proper, smooth, geometrically integral variety and not just a curve), the quotient group $\Br_{\ur}(X)/\Br_{0}(X)$ is finite. The order, however, does not seem to be recorded in the literature.
\end{remark}

\begin{corollary}\label{cor2}
Suppose that ${\del}'={\del^{\ur}}'$ (this happens, for example, if $X$ has a $K$-rational point). Then there exists a canonical perfect pairing of finite abelian groups
\begin{equation}\nonumber
\Br_{\ur}(X)/\Br_{0}(X) \times \Phi_{\ssstyle{A}}(k) \lra \Q/\Z
\end{equation}
\end{corollary}

\begin{proof}
${\del}'={\del^{\ur}}'$ implies that $d=1$. The exact sequence \ref{e2} shows that, in this case, there exists a canonical isomorphism $$\TT(A/K)\isom\Br_{\ur}(X)/\Br_{0}(X)$$ The pairing of the statement is then induced by the pairing of Corollary \ref{ttorder}.
\end{proof}

\begin{remark}
Theorem \ref{mt} shows that $c_{\ssstyle{A}}=\#\Phi_{\ssstyle{A}}(k)$ can be expressed as the product of $\#\big(\Br_{\ur}(X)/\Br_{0}(X)\big)$ and ${\del}'/{\del^{\ur}}'$. On the other hand, Theorem 1.17 in \cite{bl} expresses $\#\Phi_{\ssstyle{A}}(k)$ as the product of a term $\#\big(\Ker(\beta)/\Ima(\alpha)\big)$ and ${\del}/(q{\del^{\ur}})$ (Note that the notations for $\del$ and $\del^{\ur}$ are different in \cite{bl}). Both $\Ker(\beta)$ and $\Ima(\alpha)$ are certain subgroups of the group of Weil divisors on $X$ with support in the special fiber $X_k$. Letting $g$ be the genus of $X$, we have that $q=1$ if $\del$ divides $g-1$, and $q=2$ otherwise. As D.~Lorenzini explained to us, by \cite[Thm 7.b]{licht:duality}, $\del/q={\del}'$ so that ${\del}/(q{\del^{\ur}})={\del}'/{\del^{\ur}}$. The commutative diagram in the proof of \cite[Theorem 3]{licht:duality} then implies that since $\Br(K^{\ur})=0$, we have $\del^{\ur}={\del^{\ur}}'$. Thus we have ${\del}'/{\del^{\ur}}'={\del}'/{\del^{\ur}}={\del}/(q{\del^{\ur}})$. It follows that $\#\big(\Br_{\ur}(X)/\Br_{0}(X)\big)=\#\big(\Ker(\beta)/\Ima(\alpha)\big)$. Furthermore, comparing Cor \ref{cor2} with Cor 1.12 in \cite{bl} yields the isomorphism $$\big(\Ker(\beta)/\Ima(\alpha)\big)\isom\Hom\big(\Br_{\ur}(X)/\Br_{0}(X),\Q/\Z\big)$$ when $d=1$.
\end{remark}

\begin{remark}
The surjective map $\Phi_{\ssstyle{A}}(k)\xrightarrow{\beta}\Z/d\Z$ in Theorem \ref{mt} can be made explicit. Consider the commutative diagram
\[
\xymatrix{
\Phi_{\ssstyle{A}}(k)\ar[r]^{\beta}\ar[d]_{\alpha_1} & \Z/d\Z\\
\Hom(\TT(A/K),\Q/\Z)\ar[r]^{\alpha_2} & \Hom(\Z/d\Z,\Q/\Z)\ar[u]_{\alpha_3}}
\]
Here $\alpha_1$ is an isomorphism via the canonical perfect pairing of finite abelian groups 
$\langle\,,\,\rangle:\Phi_{\ssstyle{A}}\times \TT(A/K)\to \Q/\Z$ \cite[(4.5)]{mcc} which in turn is induced by Grothendieck's pairing $\Phi_{\ssstyle{A}}\times\Phi_{\ssstyle{A}}\to\Q/\Z$ \cite[(2.1)]{mcc}. The map $\alpha_3$ is an isomorphism as explained in the proof of Theorem \ref{mt} above. Finally, the horizontal map $\alpha_2$ is induced by the injective map $\Z/d\Z\xrightarrow{\Delta}\TT(A/K)$ in the exact sequence \ref{e2}. Note that $\Delta$ is induced by the connecting homomorphism of the long exact sequence induced by the exact sequence $0\to A\to P\to \Z\to 0$ over $K^{\ur}/K$. If $x\in\Phi_{\ssstyle{A}}(k)$, then the composition $(\alpha_2\circ\alpha_1)(x)$ is the homomorphism $\sigma:1\mapsto\langle x,\Delta(1)\rangle$. Since $\Hom(\Z/d\Z,\Q/\Z)$ is generated by $\alpha:1\mapsto{\frac{1}{d}+\Z}$, $\sigma=m\alpha$ for some  $0<m\leq{d-1}$. Then $\alpha_3(\sigma)=m$, and we let $\beta(x)=m$.
On the other hand, the injective map $\Hom(\Br_{\ur}(X)/\Br_{0}(X),\Q/\Z) \to \Phi_{\ssstyle{A}}(k)$ in Theorem \ref{mt} is induced by $\TT(A/K)\to H^1(K^{\ur}/K,P(K^{\ur}))\isom\Br_{\ur}(X)/\Br_{0}(X)$. As shown in the proof of Theorem \ref{mt} above, the map $\TT(A/K)\to H^1(K^{\ur}/K,P(K^{\ur}))$ is the restriction of the map $H^1(K,A)\to H^1(K,P)$ to $\TT(A/K)$, and the latter map is induced by the surjective map $A\to P$. Finally, the isomorphism $H^1(K^{\ur}/K,P(K^{\ur}))\isom\Br_{\ur}(X)/\Br_{0}(X)$ is induced by the map $\Br(X)\xrightarrow{\psi} H^1(K,P)$ as in \ref{pb}. The map $\psi$ is described explicitly in \cite[Rem. 2.3]{milne:brauer}.
\end{remark}

\newpage
\newcommand{\etalchar}[1]{$^{#1}$}

\end{document}